\documentclass[preprint,12pt]{article}
\usepackage[utf8]{inputenc}
\usepackage{indentfirst,csquotes}

\usepackage[thmtools-compat]{keytheorems}
\usepackage{zref-clever}
\usepackage{soul}
\zcsetup{
    cap = true,
    abbrev = false,
    nameinlink,
}

\usepackage[a4paper, margin=2.5cm]{geometry} 
\usepackage{layout}\usepackage{amsmath,amssymb,amsthm,amsfonts}
 \usepackage{comment}
\renewcommand{\leq}{\leqslant}
\renewcommand{\geq}{\geqslant}
\usepackage{graphicx}
\usepackage[normalem]{ulem}
\usepackage{tikz}

\newtheorem{theorem}{Theorem}[section]

\newtheorem{corollary}[theorem]{Corollary}
\newtheorem{lemma}[theorem]{Lemma}

\newtheorem{proposition}[theorem]{Proposition}

\newtheorem{observation}[theorem]{Observation}
\newtheorem{prblm}[theorem]{Problem}
\newtheorem{model}[theorem]{Model}

\newtheorem{remark}[theorem]{Remark}

\theoremstyle{definition}
\newtheorem{definition}{Definition}[section]

\usepackage{authblk}

\usepackage{lineno}

\title{Released  packing functions in graphs}

\title{Released packing functions in graphs} \author[1]{Pablo Fekete} \author[1]{Erica Hinrichsen} \author[1,2]{Valeria Leoni} \author[1]{M. Inés Lopez Pujato}

\affil[1]{\itshape Universidad Nacional de Rosario, Argentina}

\affil[2]{\itshape CONICET, Argentina}

 \date{}

\begin{document}
\maketitle

\maketitle \begingroup \renewcommand{\thefootnote}{} \footnotetext{ \textit{Email addresses:} fekete@fceia.unr.edu.ar (Pablo Fekete), ericah@fceia.unr.edu.ar (Erica Hinrichsen), valeoni@fceia.unr.edu.ar (Valeria Leoni), lpujato@fceia.unr.edu.ar (M. Inés Lopez Pujato) }
\endgroup

\begin{abstract}
 We introduce and start the study of a variant of packing functions in graphs. Given a graph $G$ with vertex set $V$ and nonnegative  integer vectors  $\mathbf{k}=(k_v)_{v\in V}$, $\boldsymbol\ell=(l_v)_{v\in V}$ and $\mathbf{u}=(u_v)_{v\in V}$, a function $f : V \rightarrow \mathbb{Z}_0^+$ is a  \emph{Released $( \mathbf{k}, \boldsymbol\ell, \mathbf{u})$-packing function} of $G$ if  $l_v\leq f(v)\leq u_v$ for every $v\in V$ and the sum of the values of $f$ over the closed neighborhood of  vertices $v$ with $f(v) = u_v$ is at most $k_v$.  The weight of  $f$ is the value $f(V) = \sum_{v\in V} f(v)$. 
We study the associated decision problem (RPP), which asks, given $G$,  $\mathbf{k}$, $\boldsymbol\ell$,  $\mathbf{u}$ and an integer number $x$, whether $G$ admits a Released $( \mathbf{k}, \boldsymbol\ell, \mathbf{u})$-packing function of weight at least $x$. We relate RPP to the $r$-dependent set problem, derive several  NP-hardness results,  model RPP as a compact  (polynomial in size) Integer Linear Program, and take the first steps of a polyhedral study. 
\end{abstract}

\vspace{1ex} 



\medskip 
\hrule 
\vspace{1em}

\newcommand{\N}{\mathbb N}
\newcommand{\R}{\mathbb{R}}
\newcommand{\Z}{\mathbb{Z}}

\section{Introduction }\label{sec:intro}

Over the years, packing numbers and packing functions have been used to model allocation-type problems in which the maximum possible number of necessary but obnoxious items/facilities must be allocated in certain given places. 

The following notion was recently introduced in \cite{HNV2023} as a generalization of all packing-type notions known from the literature \cite{ DHL2017, GGHR2010, HL2014,  HLS2020}.
Given $G=(V,E)$\footnote{ 
$G=(V,E)$ denotes a simple and undirected graph with vertex set $V$ and edge set $E$. When necessary, we write $V(G)$ instead of $V$ to indicate the graph of which $V$ is its vertex set.}, a nonnegative integer vector of capacities $\mathbf{k}=(k_v)_{v\in V}$ and two nonnegative integer vectors $\mathbf{u}=(u_v)_{v\in V}$ and $\boldsymbol\ell=(l_v)_{v\in V}$, a function $f : V \rightarrow \mathbb{Z}_0^+$ is a  \emph{$( \mathbf{k}, \boldsymbol\ell, \mathbf{u})$-packing function}
of $G$ if for every $v\in V$, $l_v\leq f(v)\leq u_v$ and $f(N_G[v])\leq k_v$\footnote{For $G=(V,E)$, two vertices of $V$ are \emph{adjacent} in $G$ if they share an edge of $E$.
For $v\in V$, $N_G(v)$  denotes the subset of vertices adjacent to $v$ in $G$, and $N_G[v]$, the \emph{closed neighborhood} of $v$, i.e. $N_G(v)$ together with $v$. The \emph{degree} of $v$ in $G$ is $d_G(v)= |N_G(v)|$. For a subset $S$ of $V$ and a function $f$ (or a vector $\mathbf{x}=(x_v)_{v\in V}$) defined on $V$, we denote $f(S) = \sum_{v\in S} f(v)$ ($x(S) = \sum_{v\in S} x_v$). In particular, $f(V)$ is called the weight of $f$.}. The maximum weight over all $(\mathbf{k}, \boldsymbol\ell, \mathbf{u})$-packing functions of $G$ is indicated by $L_{\mathbf{k}, \boldsymbol\ell,\mathbf{u}}(G)$.
The associated optimization problem of finding the number $L_{\mathbf{k},\ell,\mathbf{u}}(G)$  in a given graph $G$ is NP-hard even on doubly chordal graphs and polynomial time solvable on strongly chordal graphs \cite{HNV2023}.

In some real allocation-type problems,  it is not necessary to request the capacity constraint (given by $f(N_G[v])\leq k_v$) to all vertices $v$ of the graph that could model the scenario. Consider for instance an application from interference management in wireless networks: in the design of sensor networks, active nodes (vertices) cannot have too many neighbors transmitting simultaneously to avoid packet collisions. Classical packings impose this restriction on all nodes. But, if sleep-mode receiver nodes are allowed to exist, being inactive these nodes do not suffer from interference, allowing their active neighbors to transmit at a higher density, thus increasing the total system throughput.
Released packing functions capture exactly this
relaxation: the capacity constraint is required only at locations that receive
their maximum allowed load.

The paper is organized as follows. In Section \ref{sec:def}, we introduce the definition of released packing functions and contrast them with classical packing notions. In Section \ref{sec:problemdef},  by relating released packing functions with dependent sets in graphs, we derive some NP-hardness results for  the associated decision problem (RPP). In Section \ref{sec:ilpmodel}, we give  a general compact Integer Linear Programming formulation for RPP and start a polyhedral study.

\section{Released packing functions}\label{sec:def}

We introduce the following variant of packing functions in graphs. The vector $\mathbf{1}$ (resp. $\mathbf{0}$) indicates the vector with all its components equal to 1 (resp. 0), in the appropriate dimension and, for vectors  $\boldsymbol\ell=(l_v)_{v\in V}$ and $\mathbf{u}=(u_v)_{v\in V}$, the inequality  $\boldsymbol\ell \leq \mathbf{u}$ means that $l_v \leq  u_v$ for every $v$.

\begin{definition}\label{def:Releasedpacking}
Given a graph $G=(V,E)$,  nonnegative integer vectors  $\mathbf{k}=(k_v)_{v\in V}$ of capacities, and $\boldsymbol\ell=(l_v)_{v\in V}$ and $\mathbf{u}=(u_v)_{v\in V}$ with $\boldsymbol\ell \leq \mathbf{u}$, a function $f : V \rightarrow \mathbb{Z}_0^+$ that satisfies the following two conditions is a  \emph{Released  $(\mathbf{k}, \boldsymbol\ell, \mathbf{u})$-packing function}: 

\begin{itemize}
    \item[i.] for every $v\in V$, $l_v\leq f(v)\leq u_v$,
    \item[ii.] if $f(v) =u_v$ then $f(N_G[v])\leq k_v$.
\end{itemize}
The \emph{Released  $(\mathbf{k}, \boldsymbol\ell, \mathbf{u})$-packing number of $G$}, indicated by $L^R_{\mathbf{k},\boldsymbol\ell,\mathbf{u}}(G)$, is the maximum weight over all Released  $(\mathbf{k}, \boldsymbol\ell, \mathbf{u})$-packing functions of $G$, when there exists at least one. We will say that $f$ is an \emph{$L^R_{\mathbf{k},\boldsymbol\ell,\mathbf{u}}$-function } of $G$ when $f$ is a  Released  $(\mathbf{k}, \boldsymbol\ell, \mathbf{u})$-packing function of $G$ with  weight $L^R_{\mathbf{k},\boldsymbol\ell,\mathbf{u}}(G)$. 
\end{definition}

Given a graph that models an allocation-type problem like those described in the previous section, and a released packing function of it, condition ii. in Definition \ref{def:Releasedpacking} states that the capacity constraint among neighbors is enforced only for vertices (representing places) that receive their maximum allowed capacity ($u_v$ units).

\begin{remark}\label{rem:0} Released  $(\mathbf{k}, \boldsymbol\ell, \mathbf{u})$-packing functions generalize the following  well-known studied concepts in graphs:
\begin{enumerate}
    \item[i.]  Released $(\mathbf{1},\mathbf{0}, \mathbf{1})$-packing functions are just stable sets\footnote{A \emph{stable set} (or \emph{$0$-dependent set)} in a graph is a vertex subset of pairwise nonadjacent vertices. $\alpha(G)$ stands for the maximum sized stable set in the graph $G$.}, and thus $L^R_{\mathbf{1},\mathbf{0}, \mathbf{1}}(G)$ equals the well-known stable set number $\alpha(G)$. 
    \item [ii.] For  $\mathbf{k}=2 \cdot \mathbf{1}$, Released $(\mathbf{k},\mathbf{0}, \mathbf{1})$-packing functions are just 1-dependent sets\footnote{A \emph{$1$-dependent set}  in a graph is a subset of vertices that induces a subgraph  with maximum degree one.} \cite{Dessmark1993}. 
    \end{enumerate}
\end{remark}

Given a nonconnected graph $G$, it is clear that the Released  $(\mathbf{k}, \boldsymbol\ell, \mathbf{u})$-packing number of $G$ is obtained as the sum of the corresponding Released  $(\mathbf{k}, \boldsymbol\ell, \mathbf{u})$-packing numbers of its connected components. 

From their definitions, it is also clear that \begin{equation}\label{eq:eq} L_{\mathbf{k}, \boldsymbol\ell, \mathbf{u}}(G)  \leq L^R_{\mathbf{k},\boldsymbol\ell, \mathbf{u}}(G),
\end{equation}
for every graph $G$ and all vectors $\mathbf{k}$, $\boldsymbol\ell$ and $\mathbf{u}$, since every $(\mathbf{k}, \boldsymbol\ell, \mathbf{u})$-packing function of $G$ is a Released $(\mathbf{k}, \boldsymbol\ell, \mathbf{u})$-packing function of $G$.

Simple examples where  inequality (\ref{eq:eq}) is not tight, i.e., where not asking the capacity constraint to all vertices  strictly  increases the maximum, are paths and cycles\footnote{A \emph{path} with $n$ vertices, denoted by $P_n$, is a connected graph whose vertices have degree at most two. A \emph{cycle} with $n$ vertices, denoted by $C_n$, is a connected graph whose vertices have degree exactly two.}:
on the one  hand, we know from \cite{DHL2017} that 
$L_{ \mathbf{1}, \mathbf{0}, \mathbf{1}}(P_{n})=\lceil\frac{n}{3}\rceil$
and  $L_{ \mathbf{1}, \mathbf{0}, \mathbf{1}}(C_{n})=\lceil\frac{n}{3}\rceil$, 
and on the other hand it can be checked following Remark  \ref{rem:0} \textit{i.} that
$L^R_{ \mathbf{1}, \mathbf{0}, \mathbf{1}}(P_n)=\lceil\frac{n}{2}\rceil$ for  $n \geq 2$ and  $L^R_{ \mathbf{1}, \mathbf{0}, \mathbf{1}}(C_n)=\lfloor\frac{n}{2}\rfloor$  for  $n \geq 3$.
Neither $n$-wheels\footnote{An \emph{ $n$-wheel} $W_n$, is a graph formed from a cycle $C_{n}$ by adding a vertex that is adjacent to every vertex of the $C_{n}$.} for $n\geq 5$ and $\mathbf{k}=2 \cdot \mathbf{1}$ satisfy inequality (\ref{eq:eq}) by equality: on the one hand $L_{\mathbf{k},  \mathbf{0}, \mathbf{1}}(W_n)=2$, and on the other hand it can be checked following Remark \ref{rem:0} \textit{ii.}  that
$L^R_{\mathbf{k}, \mathbf{0}, \mathbf{1}}(W_n)=\lfloor\frac{2n}{3}\rfloor$.
Notice also that the difference between  $L^R_{\mathbf{k}, \mathbf{0}, \mathbf{1}}(G)$  and $L_{\mathbf{\mathbf{k}, \mathbf{0},1}}(G)$ can be arbitrarily large: for $n$-wheels,  and  $\mathbf{k}=k\cdot \mathbf{1}$ with $k\in \mathbb{Z}^+$ and $n > k\geq 3$, it is not difficult to prove that $L^R_{\mathbf{k}, \mathbf{0}, \mathbf{1}}(W_n)=n$ (the solution  where the vertices of the cycle are assigned 1, and the center  a zero is optimal), but  $L_{\mathbf{k},\mathbf{0}, \mathbf{1}}(W_n)=k$ (the solution  where any subset of $k$ vertices of the cycle are assigned 1 and the center, a zero, is optimal).

The following observations will be relevant in our subsequent study:

\begin{observation}\label{rem:kuP}

\begin{enumerate}
For a graph $G=(V,E)$ and vectors  $\mathbf{k}=(k_v)_{v\in V}$, $\boldsymbol\ell=(l_v)_{v\in V}$ and $\mathbf{u}=(u_v)_{v\in V}$ with $\boldsymbol\ell \leq \mathbf{u}$ it holds: 
\item[i.] A function $f$ is a Released  $(\mathbf{k}, \boldsymbol\ell, \mathbf{u})$-packing function  of $G$ if and only if $f$ is  a Released  $(\mathbf{k'}, \boldsymbol\ell, \mathbf{u})$-packing function of $G$, where $k'_v = \min\{k_v, u(N_G[v])\}$ for each $v\in V$, since for  all $v\in V$, 
 $f(N_G[v])\leq u(N_G[v])$.  
\item[ii.]  When $k_v=0$ and $u_v>0$ for some $v\in V$, then every Released  $(\mathbf{k}, \boldsymbol\ell, \mathbf{u})$-packing function of $G$ satisfies $f(v) <u_v$.
\item[iii.] If $l_v \neq u_v$ for every $v\in V$,  the function $f$ that assigns $u_v-1$ to each vertex $v$ is a Released  $(\mathbf{k}, \boldsymbol\ell, \mathbf{u})$-packing function of $G$. In particular, $f$ is optimal when $G$ is connected with at least 2 vertices and $\mathbf{u}=\mathbf{k}= k \cdot \mathbf{1}$  for some $k\in \mathbb{Z}^+$.
\item[iv.] For  $v \in V$ such that $l(N_G[v])> k_v$ and  $l_v=u_v$,  it does not exist a Released  $(\mathbf{k}, \boldsymbol\ell, \mathbf{u})$-packing function of $G$.
\item[v.]  Let  $v\in V$ such that $k_v < u_v$ and define $\mathbf{u}'$ and $\mathbf{k}'$ such that $u'_v = u_v-1$, $k'_v= u(N_G[v])-1$, $u'_w=u_w$ and $k'_w=k_w$ otherwise. Then, $f$ is a   Released  $(\mathbf{k}, \boldsymbol\ell, \mathbf{u})$-packing function  of $G$ if and only if  $f$ is  a Released  $(\mathbf{k'}, \boldsymbol\ell, \mathbf{u'})$-packing function of $G$. 

 \end{enumerate}
\end{observation}

\section{Problem definition and complexity}\label{sec:problemdef}

We consider the decision problem associated with released packing functions, i.e. the problem that given a graph $G$,   nonnegative  integer vectors  $\mathbf{k}=(k_v)_{v\in V}$,  $\boldsymbol\ell=(l_v)_{v\in V}$ and $\mathbf{u}=(u_v)_{v\in V}$ with $\boldsymbol\ell \leq \mathbf{u}$ and a number $x$, asks for a Released  $(\mathbf{k}, \boldsymbol\ell, \mathbf{u})$-packing function of $G$ with weight at least $x$. In this work, it is more convenient to deal with its ``optimization version'':

\begin{prblm}
\label{prob3}
{\noindent \textbf{Released packing function problem} (RPP)}
\begin{description}
\item [INSTANCE $(G, \mathbf{k}, \boldsymbol\ell, \mathbf{u})$:] A graph $G$,   nonnegative  integer vectors  $\mathbf{k}=(k_v)_{v\in V}$,  $\boldsymbol\ell=(l_v)_{v\in V}$ and $\mathbf{u}=(u_v)_{v\in V}$ with $\boldsymbol\ell \leq \mathbf{u}$.
\item [OBJECTIVE: ] To obtain $L^R_{\bf{\mathbf{k}},\boldsymbol\ell,\bf{\mathbf{u}}}(G)$. 
\end{description}
\end{prblm}

Similarly as proved in \cite{HNV2023} in the context of general packings, it is enough to consider released packing functions for null lower bounds:
\begin{lemma}\label{sololcero}
Given an instance $(G, \mathbf{k}, \boldsymbol\ell, \mathbf{u})$ of RPP for which there exists a  Released  $(\mathbf{k}, \boldsymbol\ell, \mathbf{u})$-packing function of $G$, let
$\bf{\tilde{\mathbf{u}}}=\mathbf{u}-\boldsymbol\ell$ and ${\bf{\tilde{\mathbf{k}}}}=(\tilde{ k_v})_{v\in V}$ be defined by $\tilde{ k_v}=\max\{k_v- l(N_G[v]),0\}$ for each $v\in V$. Then, 
$L^R_{\mathbf{k},\boldsymbol\ell,\mathbf{u}}(G)= L^R_{\bf{\tilde{\mathbf{k}}},\mathbf{0},\bf{\tilde{\mathbf{u}}}}(G)+l(V)$.

\end{lemma}

\begin{proof} 
Let $(G, \mathbf{k}, \boldsymbol\ell, \mathbf{u})$  be an  instance of RPP and $f$  be a Released  $(\mathbf{k}, \boldsymbol\ell, \mathbf{u})$-packing function of $G$. Consider the function $\tilde f$ defined over $V$ as $\tilde f(v)=f(v)-l_v$. Let $v\in V$. Since $l_v\leq f(v)\leq u_v$ we have $0\leq \tilde f(v)\leq u_v-l_v=\tilde u_v$ and, if $\tilde f (v)=\tilde u_v$ (thus $f(v)=u_v$), then  $\tilde f(N_G[v])=f(N_G[v])-\boldsymbol\ell(N_G[v])\leq k_v-l(N_G[v])\leq\tilde k_v$.  
It follows that $\tilde f$ is a Released $(\tilde{ \mathbf{k}},\mathbf{0}, \tilde{\mathbf{u}})$-packing function of $G$.

 Conversely, let $\tilde f$ be a Released $(\tilde k,\mathbf{0}, \tilde{\mathbf{u}})$-packing function of $G$ and  define $f$ over $V$ as $f(v)=\tilde {f}(v)+l_v$. For each $v\in V$ follows  $l_v\leq f(v)= \tilde f(v)+l_v\leq \tilde u_v+l_v=u_v$. 
  Besides, when $f (v)= \ u_v$ (which implies $\tilde f (v)= \tilde u_v$), we have:
  
\begin{itemize}
    \item 
If $u_v>\ell_v$, then $0<\tilde u_v=\tilde{f}(v)$, hence
        $0<\tilde f(v)\le\tilde f(N_G[v])\le\tilde k_v$. Thus $\tilde k_v>0$, which
        forces $\tilde k_v=k_v-\ell(N[v])$ and therefore $\ell(N[v])<k_v$.
\item If $u_v=\ell_v$, by Observation \ref{rem:kuP} \textit{iv.} and since a Released $(\mathbf{k}, \boldsymbol\ell, \mathbf{u})$-packing function of $G$ exists, it holds $\ell(N[v])\le k_v$.
\end{itemize}
In either case, $\tilde k_v=k_v-\ell(N[v])$ or equivalently   $\tilde k_v+\ell(N[v])=k_v$, and therefore 
\[
  f(N[v])=\tilde f(N[v])+\ell(N[v])\ \le\ \tilde k_v+\ell(N[v])\ =\ k_v .
\]
Hence $f$ is a Released $(k,\boldsymbol\ell,u)$-packing function of $G$. 
\end{proof}

Due to Lemma \ref{sololcero}, in the sequel we deal with instances of RPP  of the form $(G, \mathbf{k}, \mathbf{0}, \mathbf{u})$, for given $G$, $\mathbf{k}$ and $\mathbf{u}$; then, to simplify the notation, we will simply write $(G, \mathbf{k}, \mathbf{u})$.  Also, we will write   $L^R_{\mathbf{k},\mathbf{u}}(G)$ instead of $L^R_{\mathbf{k},\mathbf{0},\mathbf{u}}(G)$ and  $L_{\mathbf{k}, \mathbf{u}}(G)$ instead of $L_{\mathbf{k},\mathbf{0}, \mathbf{u}}(G)$. 
Besides, from \textit{i.} and \textit{v.} in  Observation \ref{rem:kuP}, it sufficies to consider  $u_v \leq k_v \leq u(N_G[v])$ for each $v\in V$. Then, we will say that an instance $(G, \mathbf{k}, \mathbf{u})$ of RPP is \emph{feasible}  when $u_v \leq k_v \leq u(N_G[v])$ for each $v\in V$ and such that there is at least one Released  $(\mathbf{k}, \mathbf{0}, \mathbf{u})$-packing function of $G$.

In the remainder of this section, we study the computational complexity of RPP. 
First, from Remark \ref{rem:0} \textit{i.} and the well-known NP-completeness of the Stable Set Problem\footnote{The Stable Set Problem has a graph and a number as instance and the objective is to decide if the graph has a stable set of size at least the given number.}, follows that 
    RPP is \text{NP}-complete, even with $\mathbf{k}=\mathbf{u}=\mathbf{1}$. Second, from Remark \ref{rem:0} \textit{ii.} and the well-known NP-completeness of the 1-dependent Set problem \cite{Betzler2012}, follows that
    RPP is \text{NP}-complete, even with $\mathbf{k}=2 \cdot \mathbf{1}$ and $\mathbf{u}=\mathbf{1}$.

Inspired by the relationship showed in  Remark \ref{rem:0}, we then study the complexity for general capacities. 

Recall that a \emph{$k$-dependent set} ($k\in \mathbb{Z}_0^+$)  in a graph is a subset of vertices that induces a subgraph\footnote{For $D\subseteq V$, $G[D]$ is the subgraph with vertex set $D$ and all edges of $G$ between the vertices in $D$.}  with maximum degree $k$. 
 The decision problem associated with $k$-dependent sets, the $k$-dependent set problem, has been studied in \cite{Djidjev1992}. In particular, its NP-completeness was shown
even for planar bipartite
graphs and any given $k \geq 1$ \cite{Dessmark1993}. On the positive side,  polynomial time algorithms for the $k$-dependent set problem are presented in \cite{Dessmark1993} for split graphs, cographs, trees and graphs with bounded treewidth. Here, we can state and prove:

\begin{proposition}\label{prop:bounded} 
For  fixed  $\mathbf{k}=k\cdot \mathbf{1}$ with $ k \in \mathbb{Z}^+$   and a graph $G=(V,E)$,
    solving RPP on the instance  $(G, \mathbf{k}, \mathbf{1})$ is equivalent to finding a $(k-1)$-dependent set in $G$  of maximum cardinality.
\end{proposition}

\begin{proof}
   Let $f$ be a Released $(\mathbf{k}, \mathbf{0}, \mathbf{1})$-packing function of  $G$ and consider the set  $D=\{v\in V:f(v)=1\}$ and the induced subgraph $G'=G[D]$ of $G$. Let us see that $D$ is a $(k-1)$-dependent set in $G$. 
Suppose for a contradiction that there exists a vertex $v$ in $G'$ with $d_{G'}(v) \geq k$. From the definition of $G'$ follows that $f(v)=0$ for all $v\notin D$ and $f(N_G[v])=|N_{G'}[v]|=d_{G'}(v)+1 \geq k +1$, contradicting the fact that $f$ is a Released $(\mathbf{k}, \mathbf{0}, \mathbf{1})$-packing function of $G$. In this case, $|D|=f(V)$. 
  
   Conversely, let $D \subseteq V$ be such that the subgraph $G'=G[D]$ has maximum degree $k-1$.  Define the function $f$ on $V$ such that $f(v)=1$ for each vertex in $D$ and $f(v)=0$ otherwise.     
   Take a vertex $v\in V$ with  $f(v)=1$. This implies that $v \in D$  and $v$ has degree at most $k-1$ in $G'$. Thus   $f(N_G[v])= f(N_{G'}[v])= d_{G'}(v)+1 \leq k$, implying that $f$ is a Released $(\mathbf{k}, \mathbf{0}, \mathbf{1})$-packing function of $G$. In this case,  $f(V)=|D|$. \end{proof}

\begin{theorem}\label{vertexdeletion}

 RPP is \text{NP}-complete, even for fixed  $\mathbf{k}=k\cdot \mathbf{1}$ with $ k \in \mathbb{Z}^+$ and  $k\geq 2$, $\mathbf{u}=\mathbf{1}$ and  planar bipartite graphs.
 \end{theorem}

\begin{proof}
Clearly, for fixed $\mathbf{k}=k\cdot \mathbf{1}$ with $ k \in \mathbb{Z}^+$  and $\mathbf{u}=\mathbf{1}$ and  planar bipartite graphs, RPP is in \text{NP}. 
Then, the result follows from Proposition \ref{prop:bounded} and the NP-completeness of the $(k-1)$-dependent set problem on planar bipartite graphs. 

\end{proof}

 For a fixed graph property $\Pi$, the \emph{Vertex-deletion problem} aims to find the minimum number of  vertices that must be deleted  from a given graph so that the resulting graph  satisfies $\Pi$ \cite{Lewis1980}. When $\Pi$ is the property of having bounded degree, the problem is called Bounded Degree Vertex-deletion problem (see for instance  \cite{Betzler2012}).

\begin{corollary}\label{cor:delete} For  fixed  $\mathbf{k}=k\cdot \mathbf{1}$ with $ k \in \mathbb{Z}^+$ and a graph $G$ with vertex set $V$, $L^R_{\mathbf{k}, \mathbf{1}}(G)=|V|-m_k(G)$, where $m_k(G)$ stands for the minimum number of vertices that must be deleted from $G$ to obtain a graph with maximum degree $k-1$.
 \end{corollary}

 \begin{proof}
     Follows from the duality between the $(k-1)$-dependent set problem and the Bounded Degree Vertex-deletion problem.
 \end{proof}

    Given a graph $G$ and a vector $\mathbf{k}=(k_v)_{v\in V}$, a \emph{$\mathbf{k}$-dependent} set in $G$  is  a subset of vertices that induces a subgraph  with degree at most $k_v$ in each vertex $v$ \cite{Dessmark1993}.
Under this definition, we observe that a generalization of  Proposition  \ref{prop:bounded} holds: 

\begin{observation}\label{obs:general} 
For any vector $\mathbf{k}$ of capacities, 
    solving RPP on an instance $(G, \mathbf{k}, \mathbf{1})$ is equivalent to finding a $(
    \mathbf{k}-\mathbf{1})$-dependent set in $G$  of maximum cardinality.
\end{observation}

It is also known from \cite{Dessmark1993}, that the  $
    \mathbf{k}$-dependent set problem is NP-complete on split graphs for any vector  $
    \mathbf{k}$, implying:
    
 \begin{theorem}
 RPP is \text{NP}-complete even for $\mathbf{u}=\mathbf{1}$ and  split  graphs.
 \end{theorem}   
    
\section{A general ILP model and some polyhedral results}\label{sec:ilpmodel}

In this section,  we first present a compact and general  Integer Linear Programming formulation of RPP.

Consider a feasible instance $(G, \mathbf{k}, \mathbf{u})$ of RPP. 

It is clear that any  integer-valued function $f$ defined on $V$ can be characterized by a  vector $\mathbf{x}\in \Z^{|V|}$
such that $x_v = f(v)$.
For $f$ to be a Released $(\mathbf{k}, \mathbf{0}, \mathbf{u})$-packing function of $G$, the condition
``if $f(v) =u_v$ then  $f(N_G[v])\leq k_v$'' translates to 

\begin{equation}\label{eq1} \text{`` if } \;  x_v =u_v \; \text{ then } \; x( N_G(v)) \leq k_v - u_v\text{''.} 
\end{equation}

We propose the following ILP model for  RPP:

\begin{model}
\label{model1}
\textbf{ILP model for RPP}
\par
\medskip
\noindent Maximize 
\begin{equation}
x(V)=\sum_{v \in V} x_{v}
 \label{eqn:objective}
\end{equation}
subject to
\begin{align}
  \left(u(N_G[v])-k_v \right)x_v + x(N_G(v)) 
 & \leqslant \left(u(N_G[v])-k_v  - 1\right)
    u_v + k_v  && (v\in V,\;   k_v\leq u(N_G[v])-1 ) 
\label{eqn:incoming} \\
0\leq x_{v} & \leq  u_v && (v \in V)\label{eqn:tcerouno}\\
&x_v \in \mathbb{Z} && (v \in V). 
\end{align}
\end{model}

\begin{remark}
    We do not include inequality of type \eqref{eqn:incoming} for a vertex $v$ with $k_v=u(N_G[v])$, since it would reduce to $x(N_G(v))\leq u(N_G(v))$ (already implied by inequalities in \eqref{eqn:tcerouno}).
\end{remark}

If $|V|=n$, Model \ref{model1} has only $n$ variables and at most $3n$ inequalities. Hence, it is compact in the sense that it is polynomial in the size of the input graph, making its linear relaxation suitable for resolution via commercial solvers for medium-scale instances.
\begin{lemma}
Model ~\ref{model1} solves RPP.
\end{lemma}

\begin{proof} 
For a given feasible instance $(G, \mathbf{k}, \mathbf{u})$ of RPP and a $\{0,1,\dots,\max_{v\in V}\{u_v\}\}$-valued function $f$ on $V$, let $\mathbf{x}=(x_v)_{v \in V}$ be defined as $x_v = f(v)$ for each $v \in V$. 

Let $M_v = u( N_G[v])  - k_v$. Constraint in (\ref{eqn:incoming}) for $v$ can be rewritten as $M_v x_v + x( N_G(v)) \le (M_v - 1) u_v + k_v$.
If $x_v = u_v$, it reduces to $u_v + x( N_G(v))  \le k_v$, which enforces $f(N_G[v]) \le k_v$.
Otherwise $0 \le x_v \le u_v - 1$ and  then $u_v - x_v \ge 1$. Since $f(N_G(v)) = x(N_G(v)) \le u(N_G(v))$, constraint (4) is satisfied for $v$. 

Thus, a solution $\mathbf{x}$ of Model 4.1 maps bijectively to a Released $(\mathbf{k}, \mathbf{0}, \mathbf{u})$-packing function of $G$. Finally, the objective function (\ref{eqn:objective}) maximizes the total weight.

\end{proof}

\begin{definition}
For a graph $G$ with $|V|=n$, and vectors $\mathbf{k}$ and $\mathbf{u}$, we introduce the polytope
$P(G,\mathbf{k},\mathbf{u})$ as the convex hull of vectors that satisfy all the constraints in Model \ref{model1}. 
\end{definition}

As a first step in the study of $P(G,\mathbf{k},\mathbf{u})$, we characterize the dimension of $P(G,\mathbf{k},\mathbf{u})$  in terms of the number of vertices $v$ with $u_v=0$, assuming $\mathbf{k}\geq \mathbf{1}$.
 For $v\in V$, the vector $\mathbf{e}^{v}$ indicates the \emph{characteristic vector} of $v$, i.e. the vector that has all zero components except for $v$, for which it has a 1. 

A set of vectors $\{v_0, v_1, v_2, \dots, v_k\}$ is \emph{affinely independent} if and only if the resulting set of ``difference vectors'' $\{v_1 - v_0, v_2 - v_0, \dots, v_k - v_0\}$ is linearly independent. In our proofs, we use the following equivalent definition: $\{v_0, v_1, v_2, \dots, v_k\}$ is  affinely independent if the only solution to the equation $\sum_{i=0}^k \lambda_i v_i = 0$ is $\lambda_i = 0$ for all $i$, given that the sum of the coefficients equals zero ($\sum_{i=0}^k \lambda_i = 0$).

\begin{proposition}\label{prop:dimensionofP}
Given a graph $G=(V,E)$ and vectors $\mathbf{u}\geq \mathbf{0}$ and $\mathbf{k}\geq \mathbf{1}$ such that  $u_v\leq k_v\leq u(N_G[v])$  for all $v\in V$, let  $V_0(\mathbf{u})=\{v\in V:\ u_v=0\}$. Then, the dimension of $P(G,\mathbf{k},\mathbf{u})$, 
$dim(P(G,\mathbf{k},\mathbf{u}))$, is equal to $|V|-|V_0(\mathbf{u)}|$.
\end{proposition}

\begin{proof}
    
    Let $P=P(G,\mathbf{k},\mathbf{u})$. It is clear from inequalities \eqref{eqn:tcerouno} that $P\subseteq\{x\in \R^{|V|}: x_v=0\}$ for each $v\in V_0(\mathbf{u)}$. The hyperplanes $\{x\in \R^{|V|}: x_v=0\}$  are linearly independent and thus, $dim(P)\leq |V|-|V_0(\mathbf{u)}|$. 
        It remains to exhibit $|V|-|V_0(\mathbf{u)}|+1$ affinely independent points in $P$.
        Take $v'\in V\setminus V_0(\mathbf{u)}$. We claim that $\mathbf{e}^{v'}$ satisfies inequalities in \eqref{eqn:incoming} for every $v\in V$ with $k_v\leq u(N_G[v])-1$. 
    For $x\in \R^{|V|}$ and $v\in V$, let us denote  $lhs(x,v):=\left(u(N_G[v])-k_v \right)x_v + x(N_G(v))$ and $rhs(v):= \left(u(N_G[v])-k_v-1\right)u_v+k_v$. 
Let $v\in V$ with $k_v\leq u(N_G[v])-1$. Note that $rhs(v)\geq 1$  since  $u_v\geq 0$ and $k_v\geq 1$. This implies  that $\mathbf{e}^{v'}$ satisfies inequality \eqref{eqn:incoming} associated with $v$ when $v'\notin N_G[v]$ (as $lhs(\mathbf{e}^{v'},v)=0$ in this case) and also when  $v'\in N_G(v)$ ($lhs(\mathbf{e}^{v'},v)=1$ in this case).
Finally, $lhs(\mathbf{e}^{v'},v')=u(N_G[v'])-k_{v'} \leq rhs(v')$ since  $1\leq u_{v'}\leq k_{v'}\leq u(N_G[v'])$. 
Moreover, $\mathbf{e}^{v'}$ clearly satisfies all inequalities in \eqref{eqn:tcerouno}, thus $\mathbf{e}^{v'}\in P$. Together with $\mathbf{0}$, we get a set of $|V|-|V_0(\mathbf{u})|+1$ affinely independent points in $P$, as desired, thus concluding the proof.
    \end{proof}

The strength of Model \ref{model1} lies in its validity for representing the vast number of specific problems that arise when the vectors $\mathbf{u}$ and $\mathbf{k}$ are fixed, doing so with a small number of variables and inequalities. As we have seen in the previous section, each of these specific problems can be difficult in its own right; thus, the model is not expected to be tight in every case. Below, we present two results on particular instances  that exemplify this point, demonstrating that  inequality in (7) associated with a vertex $v$ defines a facet of $P(G,\mathbf{k},\mathbf{u})$ only under restrictive hypotheses for $v$.

From now on, we continue our study fixing $\mathbf{u}=\mathbf{1}$. 
In this case,  a vector 
$\mathbf{x}=(x_v)\in \mathbb{Z}^{|V|}$ is in  $P(G,\mathbf{k},\mathbf{1})$ if and only if 
\begin{align}
(d_G(v)+1 - k_v)x_v + x( N_G(v))  \leq d_G(v) &\quad (v\in V\text{ with } k_v\leq d_G(v) ) \label{eqn: vecindad con u=1} \\ 
0\leq x_v \leq 1 \label{eqn:0,1} &\quad (v\in V).\end{align}

Clearly, as a corollary of Proposition \ref{prop:dimensionofP}, we obtain that  $P(G,\mathbf{k},\mathbf{1})$ is full dimensional when $1\leq k_v\leq d_G(v)+1$ for all $v\in V(G)$.

 The following remark establishes a useful characterization of integer points in $P(G,\mathbf{k},\mathbf{1})$:

\begin{remark}\label{rem:V1}
   For a graph $G=(V,E)$, $\mathbf{u}=\mathbf{1}$, $\mathbf{k}\geq \mathbf{1}$ and  $\mathbf{x}\in \{0,1\}^{|V|}$, let $$V_1(\mathbf{x})=\{v\in V:\ x_v=1\}.$$ Then $\mathbf{x}\in P(G,\mathbf{k},\mathbf{1})$ if and only if $|V_1(\mathbf{x})\cap N_G(v)|\leq k_v-1$ for all $v\in V_1(\mathbf{x})$.
\end{remark}

The following result describes the instances where inequality \eqref{eqn: vecindad con u=1} defines a facet of $P(G,\mathbf{k},\mathbf{1})$ with $k_v=1$.

\begin{theorem}
    \label{lem: edge ineq} Consider a graph $G=(V,E)$ and vectors $\mathbf{u}\geq \mathbf{0}$ and $\mathbf{k}\geq \mathbf{1}$ such that  $u_v\leq k_v\leq u(N_G[v])$  for all $v\in V$. Let $v_0\in V$ with $k_{v_0}=1$.
Then inequality in \eqref{eqn: vecindad con u=1} for $v=v_0$ defines a facet of $P(G,\mathbf{k},\mathbf{1})$ if and only if $d_G(v_0)=1$. 

\end{theorem}
\begin{proof} 
First, we claim that   $x_{v_0}+x_w\leq 1$ is a valid inequality for $P(G,\mathbf{k},\mathbf{1})$ for every $w\in N_G(v_0)$. Let $w\in N_G(v_0)$. 

From Remark \ref{rem:V1}, for $\mathbf x \in P\cap\{0,1\}^{|V|}$ with $\mathbf x_{v_0}=1$, we have that $|V_1(\mathbf x)\cap N(v_0)|\le k_{v_0}-1=0$, hence $x_w=0$.

Now suppose  $d_G(v_0)=1$ and  let $w$ be the only neighbor of $v_0$. For $x_{v_0} +  x_w \leq 1$ to be facet-defining, we have to prove  that there are $|V|$ affinely independent points in $P(G,\mathbf{k},\mathbf{1})$ that satisfy $x_{v_0} +  x_w \leq 1$ by equality. 
For this, let us consider the following points. For each $u \in V\setminus\{v_0,w\}$, 
let $\mathbf{x}^u\in \{0,1\}^{|V|}$  be such that $\mathbf{x}^u=\mathbf{e}^{v_0}+\mathbf{e}^u$. In addition, take $\mathbf{x}^{v_0}=\mathbf{e}^{v_0}$ and $\mathbf{x}^{w}=\mathbf{e}^{w}$. Noting that $V_1(\mathbf{x}^u)= \{v_0,u\}$, $V_1(\mathbf{x}^{v_0})=\{v_0\}$ and  $V_1(\mathbf{x}^{w})=\{w\}$, it follows from Remark \ref{rem:V1} that these $|V|$ points belong to $P(G,\mathbf{k},\mathbf{1})$. And moreover, it is not difficult to prove that they are affinely independent and each one satisfies $x_{v_0} + x_w=1$.

Conversely, suppose  inequality in (\ref{eqn: vecindad con u=1}) with $v=v_0$
defines a facet of $P(G,\mathbf{k},\mathbf{1})$ and that  $d_G(v_0) \geq 2$. From our claim, inequality $x_{v_0}+x_w\leq 1$ is valid\footnote{An inequality is valid for a set if it holds true for every single point inside that set.} for all $w\in N_G(v_0)$. Since inequality in (\ref{eqn: vecindad con u=1}) is obtained as the sum of $d_G(v_0)$ ---i.e of at least two--- valid inequalities, it does not define a facet of $P(G,\mathbf{k},\mathbf{1})$, a contradiction.
\end{proof}

From Remark \ref{rem:0}, we know that  $P(G, \mathbf{1}, \mathbf{1})$ corresponds to the well-studied Stable Set polytope $\text{STAB}(G)$ (the convex hull of all stable sets in $G$).
As seen in the previous proof, when $k_v=1$, inequality in \eqref{eqn: vecindad con u=1} is equal to the sum of edge inequalities regarding the vertex $v$, which are valid for $P(G,\mathbf{1},\mathbf{1})$. Consequentely,  if $ \tilde{P}(G,\mathbf{k},\mathbf{u}) $  denotes the polytope defined by inequalities (\ref{eqn: vecindad con u=1}) and \eqref{eqn:0,1} and $\text{ESTAB}(G)=\{x\in [0,1]^{|V|}:\ x_v+x_w\leq 1,\ vw\in E \}$ is the well-known \emph{edge relaxation} of $\text{STAB}(G)$, it follows that $\text{ESTAB}(G)\subseteq \tilde{P}(G,\mathbf{1},\mathbf{1})$.
To the best of our knowledge, $\text{ESTAB(G)}$ contains every other established relaxation of $\text{STAB(G)}$, making $\tilde{P}(G,\mathbf{1},\mathbf{1})$ a rare instance of a weaker relaxation of $\text{STAB(G)}$. While this implies a theoretical disadvantage in terms of tightness, $\tilde{P}(G,\mathbf{1},\mathbf{1})$
offers a more compact formulation that could deliver faster linear programming bounds in a Branch and Bound scheme.
For a graph $G=(V,E)$ with $|V|=n$ and $|E|=m$, $\tilde{P}(G,\mathbf{1},\mathbf{1})$ is defined by at most $3n$ inequalities, whereas $\text{ESTAB}(G)$ requires $2n+m$ inequalities. Thus, optimizing a linear function over $\tilde{P}(G,\mathbf{1},\mathbf{1})$ could prove computationally advantageous for large dense graphs.

\bigskip

We now study the case $\mathbf{k}=2\cdot \mathbf{1}$, and find  necessary and sufficient conditions on the input graph so as inequality in (\ref{eqn: vecindad con u=1})  is facet-defining for $P(G,2\cdot \mathbf{1},\mathbf{1})$.

\begin{theorem}\label{th:4.7}
    Given a graph $G=(V,E)$ and  $v\in V$ with $d_G(v)\geq 2$, inequality in \eqref{eqn: vecindad con u=1} for $v$ defines a facet of $P(G,2\cdot\mathbf{1},\mathbf{1})$ if and only if the following conditions hold:
    \begin{enumerate}
        \item[i.] $\Delta(G[N_G(v)])\leq 1$, and
        \item[ii.] $N_G(v)\setminus N_G(v')\neq \emptyset$, for all $v'\in V\setminus N_G[v]$.
    \end{enumerate}

\end{theorem}

\begin{proof}
    To simplify notation, we denote $P=P(G,2\cdot\mathbf{1},\mathbf{1})$. Let $\mathcal{F}$ be the face of $P$ defined by  inequality in \eqref{eqn: vecindad con u=1}, i.e., $\mathcal{F}=P\cap H$, where:
    \[H=\{\mathbf{x}:\ (d_G(v)-1)x_v +x(N_G(v)) = d_G(v) \}.\]

    Suppose $\mathcal{F}$ is a facet of $P$. As $P$ is a full dimensional polytope by Proposition \ref{prop:dimensionofP}, $dim(\mathcal{F})=|V|-1$. This implies that $\mathcal{F}$ is not contained in any hyperplane $H_{w,b}:=\{\mathbf{x}: x_w=b\}$, for $w\in V$ and $b\in\{0,1\}$, since otherwise $dim(\mathcal{F})\leq |V|-2$.
    
    We first prove the necessity of condition \textit{i.} By contradiction, assume $\Delta(G[N(v)])\geq 2$. Then there exist  $w_1,w_2,w_3\in N(v)$ satisfying 
    $w_1w_2,w_1w_3\in E$. Let $\mathbf{x}\in H\cap \{0,1\}^{|V|}$ such that $x_v=0$. Then $x_w=1$ for all $w\in N_G(v)$. In particular, $w_1\in V_1(\mathbf{x})$ and $\{w_2,w_3\}\subseteq V_1(\mathbf{x})\cap N_G(w_1)$. By Remark~\ref{rem:V1}, $\mathbf{x}\notin P$. Then, $x_v=1$ for all  $\mathbf{x}\in \mathcal{F}$ and thus $\mathcal{F}\subseteq H_{v,1}$, contradicting the fact that $\mathcal{F}$ is a facet. Hence, $\Delta(G[N_G(v)])\leq 1$.

    To prove the necessity of condition \textit{ii.} suppose, by contradiction, that there is a vertex $v'\in V\setminus N_G[v]$ such that $N_G(v)\subseteq N_G(v')$ and let $\mathbf{x}\in H\cap \{0,1\}^{|V|}$ with $x_{v'}=1$. If $x_v=1$, then there is (exactly) one vertex $w\in N_G(v)$ such that $x_w=1$, and then we have $\{v,v'\}\subseteq V_1(\mathbf{x})\cap N_G(w)$. Else ($x_v=0$) we have $x_w=1$ for all $w\in N_G(v)$. This implies $|V_1(\mathbf{x})\cap N_G(v')|\geq 2$, as $d_G(v)\geq 2$ and $N_G(v)\subseteq N_G(v')$. In either case, Remark~\ref{rem:V1} implies $\mathbf{x}\notin P$. Hence, $\mathcal{F}\subseteq H_{v', 0}$, contradicting the fact that $\mathcal{F}$ is a facet of $P$. Therefore, $N_G(v)\setminus N_G(v')\neq \emptyset$, for all $v'\in V\setminus N_G[v]$.

    Finally, we prove that $\mathcal{F}$ is a facet if conditions \textit{i.} and \textit{ii.} hold. 
    On the one hand, \textit{ii.} implies that for each $v'\in V\setminus N_G[v]$ there exists  $w'\in N_G(v)\setminus N_G(v')$. Consider the points $\mathbf{x}^{v'}\in\{0,1\}^{|V|}$ given by $\mathbf{x}^{v'}=\mathbf{e}^v+\mathbf{e}^{w'}+\mathbf{e}^{v'}$, for $v'\in V\setminus N_G[v]$, together with the points $\mathbf{x}^v=\sum_{w\in N_G(v)}\mathbf{e}^w$ and $\mathbf{x}^{w}=\mathbf{e}^v+\mathbf{e}^w$, for each $w\in N_G(v)$. To prove  that these $|V|$ points are affinely independent points we proceed as follows. Suppose $\sum_i\lambda_i \mathbf{x}^{i}= \mathbf{0}$ with
$\sum_i\lambda_i=0$, where the points are $\mathbf{x}^{v}$, $\{\mathbf{x}^{w}\}_{w\in N(v)}$ and
$\{\mathbf{x}^{v'}\}_{v'\in V\setminus N[v]}$. For each $v'\in V\setminus N_G[v]$, component
$v'$ equals $1$ only in $\mathbf{x}^{v'}$, then $\lambda_{v'}=0$. Besides, component $v$ is $0$ in
$\mathbf{x}^{v}$ and $1$ in every $\mathbf{x}^{w}$, hence $\sum_{w\in N_G(v)}\lambda_{w}=0$, and then
$\lambda_v=0$. Finally, for each $w\in N_G(v)$, component $w$ is $1$ only in
$\mathbf{x}^{w}$ among the remaining points, so $\lambda_w=0$. Thus all coefficients
vanish and the points exhibited are affinely independent. Clearly these points belong to $H$, and from Remark~\ref{rem:V1} and \textit{i.}, they all belong to $P$. Hence, $\mathcal{F}$ is a facet of $P$.

\end{proof}

\section{Final remarks and open questions}
We introduced Released packing functions as a new variant of packings in graphs that
relaxes the capacity constraint at selected vertex neighbourhoods. This viewpoint
extends $k$-dependent sets to an arbitrary capacity vector $k=(k_v)_{v\in V}$. Released packing functions coincide with
well-studied notions in graphs  such as stable sets and dependent sets.

 Future reseach may explore  conditions on the input instance that show that the ``released'' version of the  packing number does not exceed  the ``nonreleased'' one. For instance,  inequality (\ref{eq:eq}) in Section \ref{sec:def} is tight for 
  $\{claw, paw \}$-free  graphs  and $\mathbf{k}=2 \cdot \mathbf{1}$. A \emph{ claw} is a graph  on four vertices with three vertices of degree one and a fourth vertex that is adjacent to each of these three and a \emph{paw} is a connected graph on four vertices, that is a triangle together with a vertex of degree one.
Examples of graphs that are $\{claw, paw \}$-free are cycles and paths. Nevertheless, for the graphs pictured below  and  $\mathbf{k}=2\cdot \mathbf{1}$, it holds   $L_{\mathbf{k}, \mathbf{1}}(G) = L^R_{\mathbf{k}, \mathbf{1}}(G)$ (Corollary \ref{cor:delete} can also be applied to find the numbers $L^R_{\mathbf{k}, \mathbf{1} }(G)$); however, the first graph is $\{ claw\}$-free but not $\{ paw\}$-free, the second one is $\{ paw\}$-free but not $\{ claw\}$-free, and the last one is neither $\{ paw\}$- nor $\{ claw\}$-free.

\begin{figure}[h]
\centering
\includegraphics[scale=0.8]{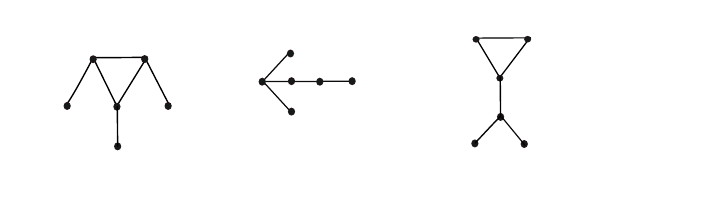}     \label{fig:fig2}
\end{figure}

Regarding the polyhedral approach to the problem, further research is needed to describe new valid inequalities for $P(G,\mathbf{k},\mathbf{u})$ that tighten the given formulation or even define facets of the polytope, with the ultimate goal of incorporating them into a Branch-and-Cut algorithm. Characterizing new facets for general (not necessarily uniform) vectors $\mathbf{u}$ and $\mathbf{k}$ emerges as a challenging task due to the complex interplay of these instance parameters, which leads to a vast number of special cases. For instance, in our preliminary attempts to generalize the conditions under which inequality in \eqref{eqn: vecindad con u=1} is facet-defining even for $\mathbf{u}=\mathbf{1}$, we found that the necessary conditions become highly restrictive and strongly dependent on the individual values of $k_w$ for each $w \in V(G)$, yet these still failed to provide sufficient conditions.

\paragraph{Acknowledgments}

This work is supported in part by  Consejo Nacional de Invesigaciones Científicas y Técnicas through project PIP 0227, and by the Universidad Nacional de Rosario through project 80020190100039UR.

\section{ Declaration of generative AI and AI-assisted technologies in the manuscript preparation process}

During the preparation of this work the authors used Google Gemini 3.1 Pro and Claude Sonnet 4.6 in order to  improve the readability and language of the manuscript and to help identify typographical errors and inconsistencies in notation. After using this tool/service, the authors reviewed and edited the content as needed and take full responsibility for the content of the published article.

\bibliographystyle{elsarticle-num} 
\bibliography{biblioarXiv} 

@article{GGHR2010,
    author  = {R. Gallant and G. Gunther and B. Hartnell and  D. Rall},
    title   ={Limited Packings in graphs},
    journal = {Discrete Applied Mathematics},
    volume  = {158},
    number  = {12},
    pages   = {1357--1364},
    year    = {2010},
    doi     = {https://doi.org/10.1016/j.dam.2009.04.014,}
    
    }

@article{Dessmark1993,
    author  = {A. Dessmark and K. Jansen and A. Lingas},
title = {The maximum k-dependent and f-dependent set problem. },
journal = { Algorithms and Computation. ISAAC 1993. Lecture Notes in Computer Science, vol 762. Springer, Berlin, Heidelberg. },
year = {1993},
URL ={https://doi.org/10.1007/3-540-57568-5\_238},
}

@article{HL2014,
    author  = { E. Hinrichsen and V. Leoni},
    title   = {$\{k\}$-packing Functions of Graphs},
    journal = {Lecture Notes in Computer Science},
    volume  = {8596},
    number  = {},
    pages   = {325--335},
    year    = {2014},
    
    }

@article{Djidjev1992,
    author  = {H. Djidjev, O. Garrido, C. Levcopoulos and A. Lingas},
    title   = {On the maximum k-dependent set problem}, 
    journal = {Technical Report, LU-CS-TR:92-91, Dept. Computer Science, Lund University, Sweden},
    volume = {},
    pages   = {},
    year    = {1992}
}

@article{HNV2023,
author   = {E. Hinrichsen and G. Nasini and N. Vansteenkiste},
title    = {On general packing functions in graphs},
 journal  = {Procedia Computer Science},
 volume   = {223},
  number   = {},
 year     = {2023},
 pages    = {367--369},
 doi      = {https://doi.org/10.1016/j.procs.2023.08.249}

}

@article{Lewis1980,
author   = {J. M. Lewis and M. Yannakakis},
title    = {The node-deletion problem for hereditary properties is NP-Complete},
 journal  = {J. Comput. System Sci.},
 volume   = {20},
  number   = {2},
 year     = {1980},
 pages    = {219–230},
 doi      = {https://doi.org/10.1016/0022-0000(80)90060-4}
 
}

@article{Betzler2012,
author   = {N. Betzler and R. Bredereck and  R. Niedermeier and J. Uhlmann},
title    = {On Bounded-Degree Vertex Deletion parameterized by treewidth},
 journal  = {Discrete Applied Mathematics},
 volume   = {160},
  number   = {1-2},
 year     = {2012},
 pages    = {53-60},
 doi      = {https://doi.org/10.1016/j.dam.2011.08.013}
 
}

@article{DHL2017,
author   = {M.P. Dobson and  E. Hinrichsen and V. Leoni},
title    = {On the complexity of the {k}-packing function problem},
 journal  = {International Transactions in  Operational Research},
 volume   = {24},
  number   = {},
 year     = {2017},
 pages    = {347–354},
 doi      ={http://doi.org/10.1111/itor.12276}
 
}

@article{HLS2020,
author   = {E. Hinrichsen and V. Leoni and M. Safe},
title    = {Labelled packing functions in graphs},
 journal  =  {Information Processing Letters},
 volume   = {154},
  number   = {},
 year     = {2020},
 pages    = {},
 doi      = {https://doi.org/10.1016/j.ipl.2019.105863}
 
}

\end{document}